\documentclass{amsart}
\usepackage{graphicx} 
\title{A description of the canonical symplectic form of the cotangent bundle of a generalized flag manifold}
\author{Juan Fervenza \and Lino Grama \and Luiz A. B. San Martin}
\usepackage[english]{babel}
\usepackage{graphicx} 
\usepackage[top=2.5cm,bottom=2.5cm,left=3cm,right=3cm]{geometry}
\usepackage[colorlinks=true,
            linkcolor=blue,
            citecolor=red,
            urlcolor=magenta]{hyperref}
\usepackage[most]{tcolorbox}
\usepackage{amssymb,amsmath,amsthm,amsfonts}
\usepackage{lipsum}
\usepackage[normalem]{ulem}
\usepackage{forest}
\usepackage{color}
\usepackage{fontawesome}
\usepackage{bm}
\usepackage{natbib} 
\usepackage{graphicx,rotating} 
\usepackage{tikz}
\usepackage{empheq}
\usepackage{wrapfig}

\allowdisplaybreaks[1] 
\usetikzlibrary{arrows.meta}
\usepackage{enumerate}
\usepackage{float}
\usepackage[T1]{fontenc}
\usepackage{faktor}
\usepackage{graphics}
\usepackage[dvips]{epsfig}
\usepackage{multirow}
\usepackage{multicol} 
\usepackage{tikz-cd}
\usetikzlibrary{arrows}

\usetikzlibrary{arrows.meta}
\tikzcdset{arrow style=tikz, diagrams={>=stealth}}
\tikzcdset{arrows={line width=0.6pt},arrow style=tikz}

\DeclareMathOperator{\Gr}{Gr}

\DeclareMathOperator{\Hom}{Hom}
\DeclareMathOperator{\SL}{SL}

\DeclareMathOperator{\proj}{proj}

\newcommand{\RE}{\mathrm{Re}}
\newcommand{\IM}{\mathrm{Im}}

\DeclareMathOperator{\id}{id}
\DeclareMathOperator{\Sp}{Sp}
\DeclareMathOperator{\pr}{pr}

\DeclareMathOperator{\SU}{SU}

\DeclareMathOperator{\Stab}{Stab}

\DeclareMathOperator{\tr}{tr}

\DeclareMathOperator{\GL}{GL}

\DeclareMathOperator{\R}{\mathbb{R}}

\DeclareMathOperator{\Ad}{Ad}

\DeclareMathOperator{\Id}{Id}
\DeclareMathOperator{\cl}{cl}

\DeclareMathOperator{\ham}{ham}

\DeclareMathOperator{\spam}{span}

\DeclareMathOperator{\SO}{SO}
\DeclareMathOperator{\LGr}{LGr}

\DeclareMathOperator{\Lie}{Lie}

\DeclareMathOperator{\g}{\mathfrak{g}}

\DeclareMathOperator{\diag}{diag}

\usepackage{icomma} 
\allowdisplaybreaks[4] 
\theoremstyle{plain}
\newtheorem{prop}{Proposition}
\newtheorem{teo}{Theorem}
\newtheorem{introtheorem}{Theorem}

\newtheorem{ex}{Example}
\newtheorem{lema}{Lemma}
\newtheorem{nota}{Notation}
\newtheorem{defi}{Definition}
\newtheorem{cor}{Corollary}
\newtheorem{obs}{Remark}

\begin{document}
\Large
\begin{abstract}
We investigate the symplectic geometry of homogeneous spaces associated with semisimple Lie groups, with particular emphasis on cotangent bundles of generalized flag manifolds. Our approach gives an explicit description of the Liouville symplectic form on these spaces in terms of the connections and curvature forms of principal bundles naturally arising from the Lie group structure. This provides a geometric link between the Lie-theoretic properties of generalized flag manifolds and the symplectic structures carried by their cotangent bundles. \vspace{5mm} 
\end{abstract}
\maketitle

\section{Introduction}

Adjoint and coadjoint orbits are classical objects in Lie theory and symplectic geometry. They carry a natural symplectic structure, the Kostant--Kirillov--Souriau (KKS) form, and arise naturally as geometric phase spaces in Hamiltonian dynamics, mechanics, and control theory (see, for instance \cite{baez2021deformations}, \cite{gasparim2021infinitesimally}, \cite{jurdjevic2020affine}, \cite{jurdjevic2022rolling}, \cite{martinez2016semisimple}). They also provide important examples of homogeneous spaces associated with semisimple Lie groups and play a fundamental role in the geometry of flag manifolds.

A particularly interesting situation arises when an adjoint orbit can be identified with the cotangent bundle of a generalized flag manifold. In this case, the same manifold carries two natural symplectic structures: the KKS form coming from the adjoint orbit and the the canonical symplectic form (the Liouville symplectic form) coming from the cotangent bundle. Understanding the relationship between these two forms is a natural question at the intersection of symplectic and homogeneous geometry.

The aim of this work is to give a concrete realization of the canonical symplectic structure of the cotangent bundle $T^{\ast}\mathbb{F}_{\Theta}$ of a generalized flag manifold $\mathbb{F}_{\Theta} =G/P_{\Theta}$, where $G$ is a non-compact semisimple Lie group and $P_{\Theta} \subset G$ is a parabolic subgroup.  To do so, we identify the cotangent bundle as a homogeneous space $T^{\ast}\mathbb{F}_{\Theta} = G/Z_{\Theta}$, where $Z_{\Theta}$ is the centralizer of an element $H_{\Theta}$ lying in the closure $\cl(\mathfrak{a}^{+})$ of the positive Weyl chamber $\mathfrak{a}^{+} \subset \mathfrak{g}$. Once we fix an Iwasawa decomposition $G=KAN$, we can realize the flag $\mathbb{F}_{\Theta}$ as a compact adjoint orbit $\mathbb{F}_{\Theta} = \Ad(K)H_{\Theta} \subset \mathfrak{g}$. The non-compact adjoint orbit $\Ad(G)H_{\Theta}$ is then identified with the cotangent bundle $T^{\ast}\mathbb{F}_{\Theta}$ \cite{gasparim2017adjoint}. Given an adjoint orbit $\mathcal{O}$, it is natural to ask when is the KKS form exact.  This question was first arised by V. Arnold in the case of $G=\SL(n, \mathbb{C})$ and refined in \cite{azada2008symplectic} in the case of a complex Lie group $G$. \begin{introtheorem} [\cite{arnold1995some}, \cite*{azada2008symplectic}]  \label{Teo arnold}
    Let $G$ be a complex Lie group and $H_{0} \in \mathfrak{g}$ a semisimple element. Let $\mathcal{O} = \Ad(G)H_{0}$ be its adjoint orbit and $\kappa$ the (complex) KKS form. Then $\RE(\kappa)$ and $\IM(\kappa)$ are real symplectic forms on $\mathcal{O}$ and 
    \begin{enumerate}[(a)]
        \item The form $\RE(\kappa)$ (resp. $\IM(\kappa)$) is exact if and only if the eigenvalues of $H_{0}$ are real (resp. purely imaginary)
        \item In either case, $\mathcal{O}$ equipped with these real forms is symplectomorphic to $T^{\ast}\mathbb{F}$.
    \end{enumerate}
\end{introtheorem}
In a recent work, \cite{Grama2026}, the second and third named authors  proved Theorem \ref{Teo arnold} in the case where $G$ is any real semisimple Lie group by obtaining an explicit description of the KKS form on cotangent bundles of maximal flags $\mathbb{F}$ $(\Theta=\emptyset)$ by composing certain linear functionals with curvatures of principal bundles related to certain homogeneous fibrations, a natural phenomenon to expect by the Chern-Weil theory. In particular, they gave a characterization of when is the KKS form on a semisimple adjoint orbit an exact symplectic form. \begin{introtheorem} (\cite{Grama2026}) \label{grama}
    Let $G=KAN$ be an Iwasawa decomposition and let $P=MAN$ be the corresponding minimal parabolic subgroup. Let $H_{0} \in \mathfrak{a}^{+}$ be a real regular element so that $\mathbb{F} = G/P =  \Ad(K)H_{0}$ and identify $T^{\ast}\mathbb{F} = G/MA = \Ad(G)H_{0}$. Then, the Liouville one form on $T^{\ast}\mathbb{F}$ is given by $$\lambda_{H_{0}, 0}^{\ast} (\omega - \omega_{K \times K})$$ where $\lambda_{H_{0}} = \langle H_{0}, \cdot \rangle \colon \mathfrak{m} \oplus \mathfrak{a} \to \mathbb{R} $ is the "height function" on the orbit. Therefore, the canonical symplectic form of $T^{\ast}\mathbb{F}$ is $$\Gamma =- \lambda_{H_{0}, 0}^{\ast} (\Omega - \Omega_{K \times K})$$ where $\Omega$ and $\Omega_{K \times K}$ are the curvatures of the connections $\omega$ and $\omega_{K \times K}$, respectively. In this expression, the component $\lambda_{H_{0}, 0}^{\ast} = \Omega_{\lambda_{H_{0}}}$ is the KKS form on $\Ad(G)H_{0}$.
\end{introtheorem}
In this paper, we generalize Theorem \ref{grama} from maximal flags to arbitrary generalized flag manifolds $\mathbb{F}_{\Theta}$.
The case $\Theta\neq\emptyset$ presents several additional difficulties that are absent in the maximal flag setting $\Theta = \emptyset$. The relevant Lie-theoretic structure is governed by the Langlands decomposition of the parabolic subalgebra 
$\mathfrak{p}_{\Theta}=\mathfrak{m}_{\Theta}\oplus\mathfrak{a}_{\Theta}\oplus\mathfrak{n}_{\Theta}^{+}$ (see Section \ref{sec:dec} for details).  Consequently, the principal bundle and connection involved in the description of $T^*\mathbb{F}_{\Theta}$ acquire a more intricate structure, and the corresponding curvature must be carefully taken into account. In particular, the horizontal and vertical components determined by the connection interact with the nontrivial isotropy directions arising from $\mathfrak{m}_{\Theta}$ and $\mathfrak{a}_{\Theta}$. This makes the passage from the Liouville symplectic form to the KKS form substantially more delicate than in the case $\Theta=\emptyset$, and requires an extension of the connection and curvature framework developed for maximal flag manifolds (cf. Section \ref{sec:KKs}). Our main result is the following. 
\\
 \begin{introtheorem} \label{Main Theorem} (Proposition \ref{liu}, Theorem \ref{can})
    Let $H_{\Theta} \in \cl(\mathfrak{a}^{+})$ so that $\mathbb{F}_{\Theta} = G/P_{\Theta} = \Ad(K)H_{\Theta}$ and identify $T^{\ast}\mathbb{F}_{\Theta}= G/Z_{\Theta} = \Ad(G)H_{\Theta}.$ If $\pi \colon G \to G/Z_{\Theta} \cong T^{\ast}\mathbb{F}_{\Theta}$ denotes the canonical projection, then the Liouville one form $\Lambda$ on $T^{\ast}\mathbb{F}_{\Theta}$ is given by \begin{equation*}
    \pi^{\ast}\Lambda =   \lambda_{H_{\Theta}, 0}(\omega_{\Delta}-\Phi^{\ast} \eta).\end{equation*} Therefore, the canonical symplectic form $\Gamma$ of $T^{\ast}\mathbb{F}_{\Theta}$ is given by $$\pi^{\ast} \Gamma = - \pi^{\ast} d\Lambda = - \lambda_{H_{\Theta}, 0} (F_{\Delta}-F_{K \times K})$$ where $F_{\Delta}$ and $F_{K \times K}$ are the curvatures of $\omega_{\Delta}$ and $\Phi^{\ast} \eta$, respectively. \\ Moreover, under the identification $T^{\ast}\mathbb{F}_{\Theta} = \Ad(G)H_{\Theta}$ the form $\Gamma$ coincides with the KKS form on the adjoint orbit.  
\end{introtheorem}
In Section \ref{sec:lag}, we develop the Lagrangian Grassmannian $\LGr_{n}$ as a guiding example, illustrating the general constructions and results developed in this paper.


\section{Decomposition of a parabolic subalgebra} \label{sec:dec}

The proof of Theorem \ref{Main Theorem} relies on a detailed understanding of the structure of the parabolic subgroup $P_{\Theta}$. The main ingredient is a careful analysis of the Langlands decomposition of its Lie algebra $\mathfrak{p}_{\Theta}=\Lie(P_{\Theta})$ (\cite{knapp1996lie}, Chapter VII). \\ Let $\mathfrak{g}$ be a semisimple Lie algebra with Iwasawa decomposition $$\mathfrak{g} = \mathfrak{k} \oplus \mathfrak{a} \oplus \mathfrak{n}^{+}.$$ We set $\Delta = \Delta(\mathfrak{g}, \mathfrak{a}) \subset \mathfrak{a}^{\ast}$ the set of restricted roots and we extract a simple system $\Sigma \subset \Delta^{+}$. By a \textit{parabolic subalgebra} of $\mathfrak{g}$ we mean any subalgebra of $\mathfrak{g}$ that contains a \textit{minimal parabolic subalgebra}. We recall that a minimal parabolic subalgebra is, modulo conjugation, given by $$\mathfrak{b} = \mathfrak{m} \oplus \mathfrak{a} \oplus \mathfrak{n}, \qquad \mathfrak{m} = Z_{\mathfrak{k}}(\mathfrak{a}) = \{ X \in \mathfrak{k} : [X, \mathfrak{a}]=0\}.$$ Every parabolic subalgebra $\mathfrak{q} \supset \mathfrak{b}$ is in correspondence with some subset $\Theta \subset \Sigma$ via $$\mathfrak{q} = \mathfrak{g}_{0} \oplus \sum _{\alpha \in \Delta^{+} \cup \langle \Theta \rangle} \mathfrak{g}_{\alpha}, \qquad \langle \Theta \rangle=\spam \Theta \cap \Delta. $$

We consider the subalgebras \begin{align*}
    \mathfrak{a}_{\Theta} &:= \bigcap_{\alpha \in \langle \Theta \rangle}  \ker \alpha = \{ X \in \mathfrak{a} : \alpha(X)=0, \forall \alpha \in \langle \Theta \rangle \} \\ \mathfrak{m}_{\Theta} &:= \underbrace{\mathfrak{m} \oplus \mathfrak{a}_{\Theta}^{\perp_{\mathfrak{a}}}}_{\mathfrak{g}_{0} \ominus\mathfrak{a}_{\Theta}} \oplus \sum _{\alpha \in \langle \Theta \rangle } \mathfrak{g}_{\alpha} \\ \mathfrak{n}_{\Theta}^{+} &:= \sum_{\alpha \in \Delta^{+} \setminus \langle \Theta \rangle} \mathfrak{g}_{\alpha}, \qquad \mathfrak{n}^{\Theta} := \sum_{\alpha \in \Delta^{+} \cap \langle \Theta \rangle} \mathfrak{g}_{\alpha}
\end{align*}
so that \begin{align} \label{langlands} \mathfrak{q} = \mathfrak{m}_{\Theta} \oplus \mathfrak{a}_{\Theta} \oplus \mathfrak{n}_{\Theta}^{+}.
\end{align}
The decomposition (\ref{langlands}) is referred to as the \textit{Langlands decomposition} of $\mathfrak{q}$. We distinguish the reductive subalgebra $$\mathfrak{z}_{\Theta} = \mathfrak{m}_{\Theta} \oplus \mathfrak{a}_{\Theta}  = \mathfrak{g}_{0} \oplus  \sum _{\alpha \in \langle \Theta \rangle} \mathfrak{g}_{\alpha}$$ which corresponds exactly to the centralizer of the element $H_{\Theta} \in \cl(\mathfrak{a}^{+})$ characterized by the property $$\Theta = \{ \alpha \in \Sigma : \alpha(H_{\Theta}) = 0\}.$$ Such $H_{\Theta}$ is called a \textit{characteristic element} of the set $\Theta$.
\begin{lema} \label{ma}
    $[\mathfrak{m}_{\Theta}, \mathfrak{a}_{\Theta}] = 0.$
\end{lema}
\begin{proof}
    Write $X \in \mathfrak{m}_{\Theta}$ as $X = Y+Z+W$ where $X \in \mathfrak{m}$, $Y \in \mathfrak{a}_{\Theta}^{\perp}$, $W \in \mathfrak{g}_{\alpha}$ with $\alpha \in \langle \Theta \rangle.$ Then for $A \in \mathfrak{a}_{\Theta}$ we have $$[A, X] = [A, Y] + [A, Z] + [A, W].$$
    The first term of the right hand side is zero because $\mathfrak{m} = Z_{K}(\mathfrak{a})$, the second one is zero as $\mathfrak{a}$ is abelian, and the third one equals $$[A, W] = \alpha(A)W = 0$$ because $\alpha \in \langle \Theta \rangle$ and $A \vert_{\langle \Theta \rangle} = 0$.
\end{proof}
\section{The KKS form and curvature} \label{sec:KKs}
We recall quickly the definition of the KKS symplectic structure on  adjoint orbits (\cite{besse2007einstein}, Chapter 8, pp. 220). Let $G$ be a semisimple Lie group with Lie algebra $\mathfrak{g}$ and Killing form $\langle \cdot, \cdot \rangle$. The $G$-adjoint orbit through an arbitrary element $H_{0} \in \mathfrak{g}$ is the set $$\mathcal{O}(H_{0}) = \Ad(G)H_{0}  = \{\Ad(g)H_{0}: g \in G\} \subset \mathfrak{g}.$$ For brevity we shall call $gH_{0}$ instead of $\Ad(g)H_{0}$ for a generic orbit element. We recall that each $X \in \mathfrak{g}$ induces a vector field $\widetilde{X}$ given by $$\widetilde{X}(gH_{0})  = \dfrac{d}{dt} \Big\vert_{t=0}\Ad(\exp tX) gH_{0} = [X, gH_{0}]$$ the so-called fundamental vector field attached to $X \in \mathfrak{g}$.

\begin{defi}
    We define the \textit{Kostant-Kirillov-Souriau} (KKS) symplectic form $\kappa \in \Omega^{2}(\mathcal{O}, \mathbb{R})$ by $$\kappa_{gH_{0}}(\widetilde{X}(gH_{0}), \widetilde{Y}(gH_{0})) :=\langle gH_{0}, [X, Y]\rangle.$$ 
\end{defi}
We refer to \cite{san2021lie}, Chapter 14, for an extensive study on the KKS form including examples. 
\begin{ex} 
    There is a Lie algebra isomorphism 
\[
\varphi \colon (\mathbb{R}^{3}, \times) \to (\mathfrak{so}(3), [\cdot, \cdot ])
\]
given by
\[
\varphi(x_{1}, x_{2}, x_{3}) =
\begin{pmatrix}
0 & -x_{3} & x_{2} \\
x_{3} & 0 & -x_{1} \\
-x_{2} & x_{1} & 0
\end{pmatrix}.
\]
If we define the inner product $$\langle X, Y \rangle = -\dfrac{1}{2} \tr (XY), \qquad X, Y \in \mathfrak{so}(3)$$ then $\varphi$ is an isometry (with respect to the standard euclidean metric in $\R^{3}$).
Moreover, $\varphi$ is $\SO(3)$-equivariant, and it restricts to a diffeomorphism $ \varphi \colon S^{2} \to M$ where $M = \Ad(\SO(3)) w$ with $w = \varphi(e_{3})$. Fix $p \in S^{2}$ and $u, v \in T_{p}S^{2}$, then $u = p \times \widetilde{u}$ and $v = p \times \widetilde{v}$. Let $F$ be the KKS form in $M$, then \begin{align*}(\varphi^{\ast} F)_{p}(u, v) &= F_{\varphi(p)}(\varphi u, \varphi v) = F_{\varphi(p)}([\varphi(p), \varphi(\widetilde{u})], [\varphi(p), \varphi(\widetilde{v})]) \\& = \langle \varphi(p), [\varphi(\widetilde{u}), \varphi(\widetilde{v})] \rangle = \langle \varphi(p), \varphi(\widetilde{u} \times \widetilde{v} ) \rangle = \langle p, \widetilde{u} \times \widetilde{v}  \rangle \\&= \langle p, u \times v \rangle \end{align*}
Therefore $\varphi^{\ast} F$ is the standard symplectic form in $S^{2}$, which is not exact.
\end{ex}
\begin{ex}
We recall the spectral theorem: If $X \in M_{n}(\mathbb{C})$ is normal (i.e. $XX^{\ast} = X^{\ast}X)$, then there exists a matrix $g \in U(n)$ s.t. $gXg^{-1} = \diag(\lambda_{1}, \hdots, \lambda_{n})$.
    Let $M = \Ad(\SU(2)) H$ where $H = \diag(i, -i) \in \mathfrak{su}(2).$ By the spectral theorem we have that $M$ consists on all matrices $X \in \mathfrak{su}(2)$ with eigenvalues $\pm i$. Let $$\varphi \colon X \in M \longmapsto \ker (X-i\Id) \in \mathbb{CP}^{1}$$
    We claim that $\varphi$ is a $\SU(2)$-equivariant diffeomorphism, indeed \begin{align*}
         v \in g \ker (X- i \Id) &\iff g^{-1}v \in \ker (X-i \Id) \\&\iff (X-i\Id)g^{-1}v = 0 \\&\iff g (X-i \Id)g^{-1} v = 0 \\& \iff (gXg^{-1}- i \Id)v = 0 \\& \iff v \in \ker (gXg^{-1} - i \Id).
    \end{align*}
    Moreover, $\varphi(H) = \mathbb{C}e_{1}$.
    Fix the canonical hermitian metric $\langle \cdot, \cdot \rangle $ in $\mathbb{C}^{2}$ so that for every line $L \in \mathbb{CP}^{1}$ we have a decomposition $\mathbb{C}^{2} = L \oplus L^{\perp}$. It is a classical result that $$T_{L} \mathbb{CP}^{1} \cong \Hom_{\mathbb{C}}(L, L^{\perp}).$$
    In particular, if $L = \mathbb{C}e_{1}$, then every map $u \colon L \to L^{\perp}$ is a matrix of the form $$u = \begin{pmatrix} 0 & -\overline{z} \\ z & 0 \end{pmatrix}$$
    and the Fubini-Study metric is given by  $$\omega_{FS}(u, v) = \IM(z\overline{w})$$
    Now, let $F$ be the KKS form on $M$. Every matrix $A \in \mathfrak{su}(2)$ is of the form $$A = \begin{pmatrix} ia & -\overline{z} \\ z & -ia \end{pmatrix} \implies [A, H] = \begin{pmatrix} 0 & 2i\overline{z} \\ 2iz & 0  \end{pmatrix} \in T_{H}M.$$
    Then $$F_{H}([A, H], [B, H]) = \langle H, [A, B] \rangle = -\tr (H \cdot [A, B])$$ We write $B = \begin{pmatrix} ib & -\overline{w} \\ w & -ib\end{pmatrix}$ to conclude that $$[A, B] = \begin{pmatrix} 2i \IM(z\overline{w}) & \ast \\ \ast & -2i\IM(z\overline{w}) \\  \end{pmatrix}$$ whence $$F_{H}([A, H], [B, H]) = 4\IM(z\overline{w})$$ which is (after suitable normalization)
    the Fubini-Study symplectic form of $\mathbb{CP}^{1}$ once we identify $\varphi_{\ast}[A, H] \cong z, \varphi_{\ast}[B, H] \cong w \in T_{L}\mathbb{CP} ^{1}$.
    \end{ex}
 We shall establish a relation between the curvature of the canonical left-invariant connection on principal bundle $Z_{\Theta} \to G \to G/Z_{\Theta}$ with the KKS form on the adjoint orbit $\Ad(G)H_{\Theta} \cong G/Z_{\Theta}$ by means of natural choices of $\Ad(Z_{\Theta})$-invariant linear functionals. 
\\To be precise, let $G$ be a semisimple Lie group with $G=KAN$ a fixed Iwasawa decomposition. At the Lie algebra level we have $\mathfrak{g}=\mathfrak{k} \oplus \mathfrak{a} \oplus \mathfrak{n}.$ Denote as usual, $\Sigma \subset \Delta^{+} \subset \mathfrak{a}^{\ast}$ the set of restricted simple roots. Fix once and for all a subset of simple roots $\Theta \subset \Sigma$ and write  $$\mathfrak{g} = \mathfrak{n}_{\Theta}^{-} \oplus \mathfrak{z}_{\Theta} \oplus \mathfrak{n}_{\Theta}^{+}.$$ We consider $\mathfrak{p}_{\Theta} = \mathfrak{z}_{\Theta} \oplus \mathfrak{n}_{\Theta}^{+}$ the parabolic subalgebra with parabolic subgroup $$P_{\Theta}=N_{G}(\mathfrak{p}_{\Theta}) = \{ g \in G : \Ad(g)\mathfrak{p}_{\Theta} \subset \mathfrak{p}_{\Theta}\}$$ for details, see \cite{wolf1998flag}. The Langlands decomposition of $\mathfrak{p}_{\Theta}$ yields a group decomposition $$P_{\Theta} = Z_{\Theta}N_{\Theta}^{+} = M_{\Theta} A_{\Theta} N_{\Theta}^{+}$$ where $Z_{\Theta}$ is the centralizer of $H_{\Theta}$ inside $G$, $$Z_{\Theta} = \{ g \in G : \Ad(g)H_{\Theta} =H_{\Theta}\}.$$ Thus we regard the $G$-adjoint orbit through $H_{\Theta}$ as a homogeneous space $$G/Z_{\Theta} \cong \Ad(G)H_{\Theta}$$ because $G$ acts transitively, through the adjoint, on $\Ad(G)H_{\Theta}$ and the stabilizer of the action is precisely $Z_{\Theta}$.

\begin{prop} \label{K stab}
       Let $H_{\Theta}$ be the characteristic element of $\Theta \subset \Sigma$. Then, the subgroup $K_{\Theta}:= P_{\Theta} \cap K \subset G$ equals the centralizer of $H_{\Theta}$ in $K$, $$K_{\Theta} = Z_{K}(H_{\Theta}) = \{k \in K : \Ad(k) H_{\Theta} = H_{\Theta}\} = K \cap Z_{\Theta}.$$
    \end{prop}
    \begin{proof}
        Let $\theta \colon G \to G$ denote the global Cartan involution, i.e., $K = \{k \in G : \theta(k) = k\}$ and $\theta (N^{+}) = N^{-}$. Then if $k \in P_{\Theta} \cap K$, then $\theta(k) = k \in P_{\Theta}^{-}$, hence $k \in P_{\Theta} \cap P_{\Theta}^{-} = Z_{\Theta}$ and we are done. 
    \end{proof}
 In view of Proposition \ref{K stab}, we have a realization of the flag manifold $\mathbb{F}_{\Theta}$ induced by $\Theta$ as an adjoint orbit of the maximal compact group $K$, $$\mathbb{F}_{\Theta} = G/P_{\Theta} \cong K/K_{\Theta} \cong \Ad(K)H_{\Theta}.$$ 

The $\mathfrak{z}_{\Theta}$-valued form $\omega \in \Omega^{1}(G, \mathfrak{z}_{\Theta})$ given by $$\omega_{g}(v) = \proj_{\mathfrak{z}_{\Theta}} (dL_{g^{-1}})_{g}v, \qquad v \in T_{g}G$$ defines a left-invariant connection on the principal $Z_{\Theta}$-bundle $G \to G/Z_{\Theta}$ and its associated distribution is $$\mathcal{H}_{g} = d(L_{g})_{e}(\mathfrak{n}_{\Theta}^{-} \oplus \mathfrak{n}_{\Theta}^{+}), \qquad g \in G.$$
\begin{nota}For any $U \in \mathfrak{g}$ we shall denote by $\widetilde{U} \in \mathfrak{X}(G/Z_{\Theta})$ the fundamental vector field associated to the left-translation of $G$ on $G/Z_{\Theta}$, that is, $$\widetilde{U}(gZ_{\Theta}) = \dfrac{d}{dt} \Big\vert_{t=0} \exp tU \cdot gZ_{\Theta}.$$ Consider $U^{r} \in \mathfrak{X}(G)$ the right-invariant vector field $$U^{r}(g) = \dfrac{d}{dt} \Big\vert_{t=0} \exp tU \cdot g = \dfrac{d}{dt} \Big\vert_{t=0} R_{g} (\exp tU) = (R_{g})_{\ast} U.$$
If $\pi \colon G \to G/Z_{\Theta}$, then $$d\pi_{g} (U^{r}(g)) = \dfrac{d}{dt} \Big\vert_{t=0} \pi (\exp tU \cdot g) = \dfrac{d}{dt} \Big\vert_{t=0} \exp tU \cdot gZ_{\Theta} = \widetilde{U}(gZ_{\Theta}).$$ 
That is \begin{equation} \label{fundamental}
\pi_{\ast}U^{r} = \widetilde{U} \circ \pi.
\end{equation}
\end{nota}
\begin{lema} \label{lema 2}
    Let $\omega \in \Omega^{1}(G, \mathfrak{z}_{\Theta})$ be the $G$-invariant connection and $v \in \mathcal{H}_{g}$ a horizontal vector. Then there exist $X \in \mathfrak{g}$ such that $X^{r}(g) = v.$
\end{lema}
\begin{proof}
    Just take $X = d(R_{g^{-1}})_{g}v \in \mathfrak{g}$.
\end{proof}

If $\langle \cdot, \cdot \rangle $ denotes the Killing form of $\mathfrak{g}$, then  $$\lambda_{H_{\Theta}} := \langle H_{\Theta}, \cdot \rangle \in \mathfrak{g}^{\ast}$$ is $\Ad(Z_{\Theta})$-invariant. Indeed, for any $z \in Z_{\Theta} = Z_{G}(H_{\Theta})$ we have  $$\lambda_{H_{\Theta}}(\Ad(z) X) = \langle H_{\Theta}, \Ad(z)X \rangle = \langle \Ad(z^{-1})H_{\Theta}, X \rangle = \langle H_{\Theta}, X \rangle   = \lambda_{H_{\Theta}}(X).$$ Moreover, we note that $\lambda_{\Theta} \in \mathfrak{z}_{\Theta}^{\ast}$ since $\mathfrak{z}_{\Theta}$ and $ (\mathfrak{n}_{\Theta} \oplus \mathfrak{n}_{\Theta}^{-})$ are orthogonal with respect to the Killing form. 

\begin{prop} Let $\lambda =  \lambda_{H_{\Theta}} \in  \mathfrak{z}_{\Theta}^{\ast}$ and $F$ be the curvature of the canonical connection form $\omega \in \Omega^{1}(G, \mathfrak{z}_{\Theta})$ on the principal $Z_{\Theta}$-bundle $G \to G/Z_{\Theta}$.
     Then, $\lambda F \in \Omega^{2}(G, \mathbb{R})$ given by  $\lambda F (X, Y):= \lambda(F(X, Y))$ descends to a two form $F_{\lambda}$ in $G/Z_{\Theta}$ by the relation $$\lambda F = \pi^{\ast} F_{\lambda}, \qquad \pi \colon G \to G/Z_{\Theta}.$$ Moreover, $F_{\lambda}$ is the negative KKS form $-\kappa$ on the adjoint orbit $\Ad(G)H_{\Theta} \cong G/Z_{\Theta}$.
\end{prop}
\begin{proof}
Let $R_{g} \colon G \to G$ be the right action of $Z_{\Theta}$ on $G$ as a principal bundle. A $k$-form $\eta$ in $G$ descends to $G/Z_{\Theta}$ if it is horizontal and right-invariant. Now, since $g \in Z_{\Theta}$ we have $$R_{g}^{\ast} (\lambda F) = \lambda(R_{g}^{\ast}F) = \lambda(\Ad(g^{-1}) F) = \lambda F$$ hence $\lambda F$ is right invariant. Since $F$ is horizontal, we have that $\lambda F$ is also horizontal and therefore $\lambda F$ descends to $G/Z_{\Theta}$. Now, let $u, v \in \mathcal{H}_{g}$ be horizontal vectors. By Lemma \ref{lema 2} there exists $X, Y \in \mathfrak{g}$ such that $X^{r}(g) = u$ and $Y^{r}(g) = v $, then by the structure equation  \begin{equation} \label{F}  F(u, v) = X^{r}(\omega(Y^{r}))_{g} - Y^{r}(\omega(X^{r}))_{g} + \omega([X, Y]^{r})_{g}\end{equation} because $[X^{r}, Y^{r}] = -[X, Y]^{r}$. First note that $$\omega(Y^{r})_{g} = \proj_{\mathfrak{z}_{\Theta}}(d(L_{g}^{-1})_g \circ d(R_{g})_{e} Y) = \proj_{\mathfrak{z}_{\Theta}} (\Ad(g^{-1}) Y)$$then  \begin{align*}X^{r}(\omega(Y^{r}))_{g} &= \dfrac{d}{dt} \Big\vert_{t=0} \omega(Y^{r})(\exp tX \cdot g) \\& = \dfrac{d}{dt} \Big\vert_{t=0} \proj_{\mathfrak{z}_{\Theta}}(\Ad(g^{-1} \exp (-tX)) Y) \\& = -\proj_{\mathfrak{z}_{\Theta}} (\Ad(g^{-1})[X, Y]). \end{align*}
Thus, Equation (\ref{F}) becomes $$F(u, v) = -\proj_{\mathfrak{z}_{\Theta}}(\Ad(g^{-1})[X, Y])$$ and therefore $$\lambda F(u, v) = -\langle H_{\Theta}, \Ad(g^{-1})[X, Y] \rangle  = - \langle \Ad(g)H_{\Theta}, [X, Y] \rangle.$$
Now, by Equation (\ref{fundamental}) we have $$\pi^{\ast} \kappa(X^{r}(g), Y^{r}(g)) = \kappa(\widetilde{X}(g), \widetilde{Y}(g)) = \langle \Ad(g)H_{\Theta}, [X, Y] \rangle = -\lambda F(X^{r}(g), Y^{r}(g))$$ hence $\pi^{\ast} \kappa = -\lambda F$ and $F_{\lambda} = -\kappa$.

\end{proof}
Now, let us consider the canonical $K$-invariant connection $\omega_{K} \in \Omega^{1}(K, \mathfrak{k}_{\Theta})$ on the principal bundle $$K_{\Theta} \to K \to K/K_{\Theta} \cong \mathbb{F}_{\Theta}.$$ 
We recall that the associated bundle $$K \times_{K_{\Theta}} Z_{\Theta}$$ carries a natural structure of principal $Z_{\Theta}$-bundle over $\mathbb{F}_{\Theta}$ (see Appendix I).
\begin{prop}\label{prop 4}
    The principal bundle $Z_{\Theta} \to G/N_{\Theta}^{+} \to \mathbb{F}_{\Theta}$ is isomorphic to the principal bundle  $Z_{\Theta} \to K \times_{K_{\Theta}} Z_{\Theta} \to \mathbb{F}_{\Theta}.$
\end{prop}
\begin{proof}
    Write $Z_{\Theta} = P_{\Theta}/N_{\Theta}^{+}$ and define \begin{align*}\psi \colon K \times_{K_{\Theta}} Z_{\Theta} &\longmapsto G/N_{\Theta}^{+} \\ [k, z] &\longmapsto kzN_{\Theta}^{+}  \end{align*}
    By the Iwasawa decomposition $G=KZ_{\Theta}N_{\Theta}^{+}$ whence $\psi$ is surjective. To see the injectivity, suppose that $\psi[k_{1}, z_{1}] = \psi[k_{2}, z_{2}]$, then \begin{align*}
        z_{2}^{-1}k_{2}^{-1}k_{1}z_{1} \in N_{\Theta}^{+}
    \end{align*}
    But since $N_{\Theta}^{+} \subset P_{\Theta}$ and $Z_{\Theta} \subset P_{\Theta}$ we have that $k_{0}:=k_{1}^{-1}k_{2} \in K \cap P_{\Theta} = K_{\Theta}$ and then $k_{2} = k_{1}k_{0}$, so $z_{2}N_{\Theta}^{+} = k_{0}^{-1}z_{1}N_{\Theta}^{+}$, i.e., $[k_{1}, z_{1}] = [k_{2}, z_{2}]$. It is also immediate that $\psi$ preserves the actions on both sides. 
\end{proof}

 By Proposition \ref{lema 6} in Appendix I we have that $\omega_{K}$ induces a connection $\omega_{K}^{+}$ on the principal bundle $$Z_{\Theta} \to K \times_{K_{\Theta}} Z_{\Theta} \to \mathbb{F}_{\Theta}$$ and by Proposition \ref{prop 4} the form $\omega_{K}^{+}$ shall be regarded as a connection form on the principal bundle $Z_{\Theta} \to G/N_{\Theta}^{+} \to \mathbb{F}_{\Theta}$ which extends $\omega_{K}$ in the sense of Remark \ref{ext}. By definition, $\widetilde{\omega}$ takes values in $\mathfrak{z}_{\Theta}$ and thus its curvature $F_{K}^{+}$ also takes values in $\mathfrak{z}_{\Theta}$.

\begin{lema}The curvature $F_{K}^{+} \in \Omega^{2}(G/N_{\Theta}^{+}, \mathfrak{z}_{\Theta})$ of the extended connection $\omega_{K}^{+}$ takes values in $\mathfrak{m}_{\Theta}$.
\end{lema}
\begin{proof}

For $kz \in K \times_{K_{\Theta}} Z_{\Theta} \cong G/N_{\Theta}^+$ we have (cf. Thm \ref{most})  $$(F_{K}^{+})_{[k,z]} = \Ad(z^{-1}) (F_{K})_{k}$$ and since $Z_{\Theta} = M_{\Theta}A_{\Theta}$ it suffices to show that  $\Ad(M_{\Theta})\mathfrak{k}_{\Theta} \subset \mathfrak{m}_{\Theta}.$
If $\mathfrak{k}_{\Theta} \subset \mathfrak{m}_{\Theta}$, then the proof shall be concluded. To see this, let $X \in \mathfrak{k}_{\Theta} = \mathfrak{k} \cap \mathfrak{z}_{\Theta}$ (cf. Prop. \ref{K stab}), and write $X = Y+Z$, with $Y \in \mathfrak{m}_{\Theta}$ and $Z \in \mathfrak{a}_{\Theta} \subset \mathfrak{s}$. Apply the Cartan involution $\theta X = \theta Y-Z$. Since $\theta(\mathfrak{m}_{\Theta}) \subset \mathfrak{m}_{\Theta}$ we have that $\theta Y \in \mathfrak{m}_{\Theta}$ and since $X \in \mathfrak{k}$ we have $\theta X = X$; therefore $X = Y \in \mathfrak{m}_{\Theta}$.
\end{proof}
\begin{obs} Since $\mathfrak{m}_{\Theta} \cap \mathfrak{a}_{\Theta} = \{0\}$ we have that the curvature in the $\mathfrak{a}_{\Theta}$-direction vanishes. We note that $H_{\Theta} \in \mathfrak{a}_{\Theta}$.
\end{obs}
\begin{lema}
    We have a principal bundle $$AN^{\Theta} \to G/N_{\Theta}^{+} \to K.$$ 
\end{lema}
\begin{proof}
    The proof relies in the fundamental fact: If $H_{1} \subset H_{2} \subset G$ are closed subgroups of a Lie group $G$ and $H_{1}$ is normal in $H_{2}$, then $$H_{2}/H_{1} \to G/H_{1} \to  G/H_{2}$$ is a principal fibre bundle (see \cite{san2021lie}, Proposition 13.26). We set $H_{2}=AN$ and $H_{1} = N_{\Theta}^{+}$. Let's prove that $N_{\Theta}^{+}$ is normal in $AN^{+}$. Since $AN^{+}$ and $N_{\Theta}^{+}$  are connected, it suffices to show that $\mathfrak{n}_{\Theta}^{+}$ is an ideal in $\mathfrak{a} \oplus \mathfrak{n}^{+}$, and since $\mathfrak{a}$ is abelian, it suffices to show that $\mathfrak{n}_{\Theta}^{+}$ is an ideal in $\mathfrak{n}^{+}$. Let us write $$[\mathfrak{n}_{\Theta}^{+}, \mathfrak{n}^{+}]=\Big[\sum_{\alpha \in \Delta^{+} \setminus \langle \Theta \rangle} \mathfrak{g}_{\alpha}, \sum_{\beta \in \Delta^{+}} \mathfrak{g}_{\beta} \Big] \subset \sum_{\alpha, \beta} \mathfrak{g}_{\alpha+\beta}$$ so that we have to show $\alpha + \beta \notin \langle \Theta \rangle $ for $\alpha \in \Delta^{+} \setminus \langle \Theta \rangle$ and $\beta \in \Delta^{+}$. As a linear combination of simple roots, both $\alpha$ and $\beta$ have non-negative integer coefficients, so that $\alpha = (\hdots) + a_{i}\alpha_{i}$ where $\alpha_{i} \in \Sigma \setminus \Theta$ and $a_{i} >0$. But then, $\alpha + \beta = (\hdots) + (a_{i}+b_{i})\alpha_{i}$ where $a_{i}+b_{i} >0$ (because $b_{i} \geq 0)$, whence $\alpha + \beta \notin \langle \Theta \rangle.$ 
    Now, since $AN^{\Theta} \cong AN/N_{\Theta}^{+}$ and $G/AN \cong K$ we have the result. 
\end{proof}
Unlike the curvature, the extended connection $\omega_{K}^{+}$ has a $\mathfrak{a}_{\Theta}$ component as described in the following Proposition. 
\begin{prop} \label{ma invariant}
     For any $U \in \mathfrak{g}$ let $\widetilde{U} \in \mathfrak{X}(G/N_{\Theta}^{+})$ be the induced vector field, and denote $x_{0}:=eN_{\Theta}^{+}$ the origin of $G/N_{\Theta}^{+}$. Then, if $g = khn \in KAN$ we have $$\lambda_{H_{\Theta}}(\omega_{K}^{+}(\widetilde{U}(gx_{0}))) = \langle \Ad(k)H_{\Theta}, U \rangle.$$
\end{prop}
\begin{proof}
First we prove the formula at the origin $x_{0} = eN_{\Theta}^{+}$. To do so, decompose the Lie algebra $\g$ as $$\g = \mathfrak{k} \oplus \mathfrak{a} \oplus \mathfrak{n}^{\Theta} \oplus \mathfrak{n}^{+}_{\Theta}$$ and write $U = U_{\mathfrak{k}} + U_{\mathfrak{a}} + U_{\mathfrak{n}^{\Theta}} + U_{\mathfrak{n}_{\Theta}^{+}}$. It is clear that $\widetilde{U}(x_{0}) = U_{\mathfrak{k}} + U_{\mathfrak{a}} + U_{\mathfrak{n}^{\Theta}}$ and $$\omega_{K}^{+}(\widetilde{U}(x_{0})) = (\omega_{K})_{e}(U_{\mathfrak{k}}) + \theta_{e}(U_{\mathfrak{a}} + U_{\mathfrak{n}^{\Theta}}) = (U_{\mathfrak{k}})_{\mathfrak{k}_{\Theta}} + U_{\mathfrak{a}} + U_{\mathfrak{n}^{\Theta}}.$$
Now, since $\mathfrak{k}_{\Theta}$ and $\mathfrak{n}^{\Theta}$ are orthogonal to $\mathfrak{a}_{\Theta}$ we have $$\lambda_{H_{\Theta}}(\omega^{+}_{K}(\widetilde{U}(x_{0}))) = \langle H_{\Theta}, U_{\mathfrak{a}} \rangle = \langle H_{\Theta}, U \rangle .$$ This is the desired formula over the origin. We now proceed to compute the formula on points of the form $z x_{0}$, with $z \in AN^{\Theta} \subset Z_{\Theta}$ that is, in the fiber over the origin. Note that $\widetilde{U}(zx_{0}) = (R_{z})_{\ast} \widetilde{U}(x_{0})$ whence $$\omega_{K}^{+}(\widetilde{U}(zx_{0})) = \omega_{K}^{+}( (R_{z})_{\ast} \widetilde{U}(x_{0})) = \Ad(z^{-1}) \omega_{K}^{+}(\widetilde{U}(x_{0})),$$ and since $\lambda_{H_{\Theta}}$ is $Z_{\Theta}$-invariant, we have $$\lambda_{H_{\Theta}}(\omega_{K}^{+}(\widetilde{U}(zx_{0}))) = \lambda_{H_{\Theta}}(\omega_{K}^{+}(\widetilde{U}(x_{0}))) = \langle H_{\Theta}, U \rangle.$$ This proves the desired formula for any point on the fibre over the origin. To finish we must check the formula for points of the form $k x$,  where $k \in K$ and $x$ is in the fibre over the origin. Note that $$\widetilde{U}(kx) = (L_{k})_{\ast} \widetilde{W}(x), \qquad W = \Ad(k^{-1}) U,$$ whence since $\omega_{K}^{+}$ is $K$ left-invariant we obtain $$\lambda_{H_{\Theta}}(\omega_{K}^{+}(\widetilde{U}(kx))) = \langle H_{\Theta}, \omega_{K}^{+}((L_{k})_{\ast} \widetilde{W}(x)) \rangle = \langle H_{\Theta}, W \rangle = \langle H_{\Theta}, \Ad(k^{-1})U \rangle= \langle \Ad(k)H_{\Theta}, U \rangle.$$
\end{proof}
\section{Specializing to the partial flag $\mathbb{F}_{\Theta} \times \mathbb{F}_{\Theta}^{\ast}$}
From now on we shall identify the dual flag (see Appendix II) $$\mathbb{F}_{\Theta}^{\ast}  = G/P_{\Theta}^{-} = G/Z_{\Theta}N_{\Theta}^{-}.$$ Let us consider now the identification \begin{align*} (G \times G)/(N_{\Theta}^{+} \times N_{\Theta}^{-}) & \stackrel{\cong}\longrightarrow G/N_{\Theta}^{+} \times G/N_{\Theta}^{-} \\ (g_{1}, g_{2}) \cdot (N_{\Theta}^{+} \times N_{\Theta}^{-}) & \longmapsto (g_{1}N_{\Theta}^{+}, g_{2} N_{\Theta}^{-})  \end{align*} so that $Z_{\Theta} \times Z_{\Theta}$ acts on $(G \times G)/(N_{\Theta}^{+} \times N_{\Theta}^{-})$ naturally on the right  $$(z_{1}, z_{2}) \cdot [(g_{1}, g_{2}) \cdot (N_{\Theta}^{+} \times N_{\Theta}^{-}) ] := (g_{1}z_{1}, g_{2}z_{2}) \cdot (N_{\Theta}^{+} \times N_{\Theta}^{-}),$$ and we have a principal bundle 
 \begin{equation} \label{principal} Z_{\Theta} \times Z_{\Theta} \to (G \times G)/(N_{\Theta}^{+} \times N_{\Theta}^{-}) \to \mathbb{F}_{\Theta} \times \mathbb{F}_{\Theta}^{\ast}.\end{equation} Observe that  $G$ acts diagonally on $(G \times G)/(N_{\Theta}^{+} \times N_{\Theta}^{-}) \cong G/N_{\Theta}^{+} \times G/N_{\Theta}^{-}$ and the orbit through $\xi = (eN_{\Theta}^{+}, eN_{\Theta}^{-})$ is diffeomorphic to $G$, invariant by the right action of $\Delta(Z_{\Theta})=\{(z, z): z \in Z_{\Theta}\} \cong Z_{\Theta}$ and projects onto the open orbit $G \cdot (o, o_{w}) \subset \mathbb{F}_{\Theta} \times \mathbb{F}_{\Theta}^{\ast}$. Hence, the restriction of (\ref{principal}) to the orbit $G \cdot \xi $ is isomorphic to the principal bundle $  Z_{\Theta} \to G \to G/Z_{\Theta}$. \\ 
 The Iwasawa decompositions of $G = KAN^{+} = KAN^{-}$ induces a Iwasawa decomposition of $G \times G$ $$G \times G = (K \times K) \cdot (A \times A) \cdot (N^{+} \times N^{-}).$$
 Thus we can consider as the  $K \times K$-invariant connection $$\omega_{K \times K} \in \Omega^{1}(K \times K, \mathfrak{k}_{\Theta} \times \mathfrak{k}_{\Theta})$$ on the principal bundle $K_{\Theta} \times K_{\Theta} \to K \times K \to \mathbb{F}_{\Theta} \times \mathbb{F}_{\Theta}^{\ast}$. 
\\ 
Analogously as before, the $K$-invariant connection $\omega_{K} \in \Omega^{1}(K, \mathfrak{k}_{\Theta})$ can be extended to a connection $\omega_{K}^{-} \in \Omega^{1}(G/N_{\Theta}^{-}, \mathfrak{z}_{\Theta})$ defined on the principal bundle $$Z_{\Theta} \to G/N_{\Theta}^{-} \to \mathbb{F}_{\Theta}^{\ast}.$$ The following proposition is straightforward by our construction of the extended connection (Remark \ref{ext}).

\begin{prop}
The connection $\eta \in \Omega^{1}(G/N_{\Theta}^{+} \times G/N_{\Theta}^{-}, \mathfrak{z}_{\Theta} \times \mathfrak{z}_{\Theta})$ given by 
    $$ \eta(u, v) = (\omega_{K}^{+}(u), \omega_{K}^{-}(v)), \qquad u, v \in T(G/N_{\Theta}^{+} \times G/N_{\Theta}^{-}), $$ is an extension of the $K \times K$--invariant connection $\omega_{K \times K} \in \Omega^{1}(K \times K, \mathfrak{k}_{\Theta} \times \mathfrak{k}_{\Theta})$.
\end{prop}
Let us write $(x_{0}, y_{0}) = (eN_{\Theta}^{+}, eN_{\Theta}^{-})$ the origin of $G/N_{\Theta}^{+} \times G/N_{\Theta}^{-}$ and remember that the diagonal orbit $$G \cdot (x_{0}, y_{0}) = \{(gx_{0}, gy_{0}) : g \in G\} \subset G/N_{\Theta}^{+} \times G/N_{\Theta}^{-}$$ is diffeomorphic to $G$ as the stabilizer of the diagonal action is $N_{\Theta}^{+} \cap N_{\Theta}^{-} = \{e\}$. Let us denote this diffeomorphism by $\Phi \colon G \stackrel{\cong}{\to} G\cdot (x_{0}, y_{0})$. Now, for every $U \in \mathfrak{g}$ we have  $$\Phi_{\ast} U^{r}(g) = (\widetilde{U}(gx_{0}), \widetilde{U}(gy_{0}))$$ hence, by definition, \begin{equation} \label{extended} (\Phi^{\ast} \eta)(U^{r}(g)) = (\omega_{K}^{+}(\widetilde{U}(gx_{0})), \omega_{K}^{-}(\widetilde{U}(gy_{0})) ).\end{equation}
\\ Now for every $H_{1}, H_{2} \in \mathfrak{a}_{\Theta}$ we can consider the linear functional $\lambda_{H_{1}, H_{2}} \in (\mathfrak{z}_{\Theta} \times \mathfrak{z}_{\Theta})^{\ast}$  given by $$\lambda_{H_{1}, H_{2}}(u, v) = \langle H_{1}, u\rangle+ \langle H_{2}, v\rangle , \qquad u, v \in \mathfrak{z}_{\Theta}. $$
Since $Z_{\Theta} = Z_{G}(\mathfrak{a}_{\Theta})$, we have that $\lambda_{H_{1}, H_{2}}$ is ($Z_{\Theta} \times Z_{\Theta}$)--invariant. We then can consider the composite $\lambda_{H_{1}, H_{1}}(\Phi^{\ast} \eta) \in \Omega^{1}(G, \mathbb{R})$. Therefore, equation (\ref{extended}) together with Proposition \ref{ma invariant} yields  \begin{equation} \label{KK}
    \lambda_{H_{1}, H_{1}}(\Phi^{\ast} \eta  (U^{r}(g))) = \langle \Ad(k_{+})H_{1}, U\rangle + \langle \Ad(k_{-})H_{2}, U \rangle,
\end{equation} where $g = k_{+}a_{+}n_{+} \in KAN^{+}$ and $g = k_{-}a_{-}n_{-} \in KAN^{-}.$
\\ As proved in \cite{gasparim2017adjoint}, Theorem 2.1, the $G=KAN^{+}$ decomposition induces a $K$-equivariant fiberwise diffeomorphism \begin{equation} \label{iota}\iota \colon \underbrace{\Ad(G)H_{\Theta}}_{\bigcup_{k \in K} \Ad(k) (H_{\Theta}+\mathfrak{n}_{\Theta}^{+})} \owns \Ad(k)(H_{\Theta}+X) \longmapsto \langle \Ad(k)X, \cdot \rangle \in T^{\ast}_{ko}\mathbb{F}_{\Theta} \end{equation}
where $o = eK_{\Theta}$ is the origin of the flag $\mathbb{F}_{\Theta}$. Hence, we shall write a generic linear functional $\alpha \in T_{ko}^{\ast}\mathbb{F}_{\Theta}$ as $$\alpha = \Ad(k)(H_{\Theta} + X), \quad X \in \mathfrak{n }_{\Theta}^{+}.$$ \begin{obs} \label{impobs}
        Let $p \colon T^{\ast}M \to M$ be a cotangent bundle with $\Gamma$ the canonical symplectic form $\Gamma = -d\Lambda$ where $$\Lambda_{\alpha}(v) = \alpha (p_{\ast} v), \qquad v \in T_{\alpha} (T^{\ast}M)$$ is the Liouville form. Every diffeomorphism $g \colon M \to M$ lifts to a symplectomorphism $g^{\#} \colon (T^{\ast}M, \Gamma) \to (T^{\ast}M, \Gamma)$ given by $g^{\#}(\alpha) = \alpha \circ (dg^{-1})_{g(x)},  \alpha \in T_{x}^{\ast}M$. Thus, an action  $\psi \colon G \times M \to  M$ lifts to a symplectic action  $\psi^{\#} \colon G \times T^{\ast}M \to T^{\ast}M$, $$\psi^{\#}(g, \alpha) := \psi(g)^{\#}(\alpha).$$
        The infinitesimal generator of $\psi^{\#}$ is the usual cotangent lift of vector fields $ X  \in  \mathfrak{g} \to \hat{X}^{\#} \in \mathfrak{X}(T^{\ast} M)$, where $\hat{X} \in \mathfrak{X}(\mathbb{F}_{\Theta}) $ is vector field induced by the action $\psi$ of $G$ on $M$, that is, $\widehat{X}(x) = d(\psi_{x})_{e}X$. Now, it is known that \begin{equation} \label{ham} \widehat{X}^{\#} = \ham (\Gamma, h_{X})\end{equation} where $h_{X}(\alpha) = \alpha (X_{x}), \alpha \in T_{x}^{\ast} M$. Hence, the moment map $\mu \colon T^{\ast}M \to \mathfrak{g}^{\ast}$ is given by \begin{align} \label{liouvillemoment}\mu(\alpha)(X) = h_{X}(\alpha) = \alpha (\pi_{\ast} \widehat{X}^{\#}(\alpha)) = \Lambda(\widehat{X}^{\#})(\alpha).\end{align}
        Moreover, since $\widehat{X}(x) = d(\psi_{x})_{e}X$, we can also write \begin{equation} \label{dpsi} \mu(\alpha) = \alpha \circ d(\psi_{x})_{e} \iff \mu = (d(\psi_{\pi (\cdot)}))_{e}^{\ast}  \iff \mu \vert_{T_{x}M^{\ast}} = (d(\psi_{x})_{e})^{\ast}.\end{equation}
        \end{obs}
        Now, $G$ acts on $\mathbb{F}_{\Theta}$ by left translations, and such action (by Remark \ref{impobs}) lifts to a symplectic action of $G$ on $T^{\ast}\mathbb{F}_{\Theta}$. The moment map of this action shall be described in the following Proposition. 
\begin{prop} \label{mom}
The moment $\mu \colon T^{\ast} \mathbb{F}_{\Theta} \to \mathfrak{g}^{\ast}$ of the symplectic action is $$\mu(\alpha)(U) = \langle \Ad(k)X, U\rangle, \qquad \alpha = \Ad(k)(H_{\Theta}+X), U \in \mathfrak{g}.$$
\end{prop}
\begin{proof} Let $\psi \colon G \times \mathbb{F}_{\Theta} \to \mathbb{F}_{\Theta}$ be the left action. By equation (\ref{dpsi}) we have $$\mu(\alpha)(U) = \alpha(d(\psi_{x})_{e}U) = \alpha(\hat{U}(x))$$ now, by (\ref{iota}) $$\alpha(\hat{U}(x)) = \langle \Ad(k)X, \hat{U}(x) \rangle $$ and $\hat{U}(x) \in T_{ko}\mathbb{F}_{\Theta} = \Ad(k)(\mathfrak{n}_{\Theta}^{-}).$ To see that in the formula we can drop the tilde, just note that $$\mathfrak{g} =  \Ad(k)(\mathfrak{z}_{\Theta}) \oplus^{\perp} (\Ad(k)(\mathfrak{n}_{\Theta}^{-}) \oplus \Ad(k)  (\mathfrak{n}_{\Theta}^{+}) )$$, hence   $$\langle \Ad(k)X, U\rangle = \langle \Ad(k)X, \pr_{\Ad(k)\mathfrak{n}_{\Theta}^{-}}(U) \rangle = \langle \Ad(k)X, \hat{U}(x) \rangle$$ for $x = ko \in \mathbb{F}_{\Theta}.$
\end{proof}
\begin{cor} \label{coro}
Let $\omega \in \Omega^{1}(G, \mathfrak{z}_{\Theta})$ be the canonical $G$-invariant connection on bundle $Z_{\Theta} \to G \to G/Z_{\Theta}$, and let $\Phi^{\ast}\eta \in \Omega^{1}(G, \mathfrak{z}_{\Theta} \times \mathfrak{z}_{\Theta}) $ be the extension of the $K \times K$-invariant connection.  Then for $\alpha = \Ad(g)H_{\Theta} =  \Ad(k)(H_{\Theta} + X) \in T^{\ast}_{ko} \mathbb{F} \cong \Ad(k)(H_{\Theta} + \mathfrak{n}^{+}_{\Theta})$, we have $$\Lambda_{\alpha}(U^{\#}(\alpha)) = \lambda_{H_{\Theta}}(\omega(U^{r}(g))) - \lambda_{H_{\Theta},0} (\Phi^{\ast} \eta (U^{r}(g))).$$ 
\end{cor}
\begin{proof} Equation (\ref{liouvillemoment}) and Proposition \ref{mom} yields $$\Lambda_{\alpha}(U^{\#}(\alpha)) = \langle \Ad(k)X, U\rangle.$$
But now, $$\langle \Ad(k)X, Z \rangle =  \langle \Ad(k)H_{\Theta} + \Ad(k)X, U\rangle  - \langle \Ad(k)H_{\Theta}, U \rangle .$$ Note that  $$\omega(U^{r}(g)) = \pr_{\mathfrak{z}_{\Theta}} d(L_{g^{-1}})_{g} U^{r}(g) = \pr_{\mathfrak{z}_{\Theta}} d(L_{g^{-1}})_{g} d(R_{g})_{e} U = \pr_{\mathfrak{z}_{\Theta}} \Ad(g^{-1})U$$ whence $$\lambda_{H_{\Theta}}(\omega(U^{r}(g))) = \langle \Ad(g)H_{\Theta}, U\rangle.$$ Therefore, if  $\Ad(g)H_{\Theta} = \alpha$, then $$\lambda_{H_{\Theta}}(\omega(U^{r}(g))) = \langle \Ad(k)H_{\Theta} + \Ad(k)X, U \rangle .$$
Now, the last term was computed in terms of the connection $\omega_{K \times K}$ as in (\ref{KK}), and we get $$\lambda_{H_{\Theta}, 0}(\Phi^{\ast} \eta (U^{r}(g))) = \langle \Ad(k)H_{\Theta}, Z \rangle$$
\end{proof}
Let $\omega_{\Delta} \in \Omega^{1}(G, \mathfrak{z}_{\Theta} \times \mathfrak{z}_{\Theta})$ the form given by $\omega_{\Delta}(v) = (\omega(v), \omega(v)) \in \mathfrak{z}_{\Theta} \times \mathfrak{z}_{\Theta}.$ Hence we have the formula   \begin{equation} \label{lambda} \Lambda(U^{\#}(\alpha))= \big(\lambda_{H_{\Theta}, 0}(\omega_{\Delta} - \Phi^{\ast}\eta) \big)(U^{r}(g)) \end{equation} for $\alpha = \Ad(g)H_{\Theta}$. 
\\
We can consider an action $a \colon G \times T^{\ast}\mathbb{F}_{\Theta} \to T^{\ast} \mathbb{F}_{\Theta}$ by integrating the infinitesimal action $q \colon \mathfrak{g} \to \mathfrak{X}(T^{\ast}\mathbb{F}_{\Theta})$ given by $$q(X) = \begin{cases}
    X^{\#}, \quad X \in \mathfrak{k} \\ X^{\#}+ V_{X}, \quad X \in \mathfrak{s}
\end{cases} $$
where $V_{X}$ is a vertical vector field on $T^{\ast}\mathbb{F}_{\Theta}$. In the paper \cite{gasparim2017adjoint}, Section 2.3, the authors proved that the moment map of this action establishes a $G$-invariant diffeomorphism  $$ \ T^{\ast}\mathbb{F}_{\Theta} \to \Ad(G)H_{\Theta} .$$ 
In summary, we have the following description of the fundamental vector field $\widehat{U}(gZ_{\Theta}) = \pi_{\ast}U^{r}(g)$ in terms of vector fields defined on the cotangent bundle $T^{\ast}\mathbb{F}_{\Theta}$:
 \begin{equation} \label{ks} \pi_{\ast} U^{r}(g) = U_{\mathfrak{k}}^{\#}(\alpha) + U^{\#}_{\mathfrak{s}}(\alpha) + V_{U_{\mathfrak{s}}}(\alpha) 
\end{equation}
where $\alpha = \Ad(g)H_{\Theta} =gZ_{\Theta}$.
\begin{prop} \label{liu} The Liouville form is given by \begin{equation}
    \label{lio}\pi^{\ast}\Lambda =   \lambda_{H_{\Theta}, 0}(\omega_{\Delta}-\Phi^{\ast} \eta). \end{equation} 
\end{prop}
\begin{proof}
    Denote by $p \colon T^{\ast}\mathbb{F}_{\Theta} \to \mathbb{F}_{\Theta}$ the usual projection.  Note that if $v \in T_{\alpha} T^{\ast}\mathbb{F}_{\Theta}$ is vertical, then $\Lambda(v) = \alpha(p_{\ast}v) = 0$ as $p_{\ast}v=0$. Hence by \ref{ks} $$\Lambda(\pi_{\ast} U^{r}(g)) = \Lambda(U_{\mathfrak{k}}^{\#}(\alpha) + U^{\#}_{\mathfrak{s}}(\alpha) + V_{U_{\mathfrak{s}}}(\alpha)) = \Lambda(U^{\#}(\alpha)).$$ The result now follows by Equation \ref{lambda}. 
\end{proof}
Now we are able to state and prove the main result of this work. 
 \begin{teo} \label{can}
    The canonical symplectic form $\Gamma \in \Omega^{2}(T^{\ast} \mathbb{F}_{\Theta}, \mathbb{R})$ is given by $$\pi^{\ast} \Gamma = - \lambda_{H_{\Theta}, 0} (F_{\Delta}-F_{K \times K})$$ where $F_{\Delta}$ and $F_{K \times K}$ are the curvatures of $\omega_{\Delta}$ and $\Phi^{\ast} \eta$, respectively. \\ Moreover, if $\kappa$ denotes the KKS form on $\Ad(G)H_{\Theta} \cong T^{\ast}\mathbb{F}_{\Theta}$ we have that $$\pi^{\ast} \Gamma = -\lambda_{H_{\Theta}}F \iff \Gamma =- F_{\lambda_{H_{\Theta}}}=\kappa.$$
    Therefore the KKS form is exact. 
\end{teo}
\begin{proof}
    If we differentiate Equation (\ref{lio}) we obtain (by the structure equation) $$\pi^{\ast} d\Lambda = \lambda_{H_{\Theta}, 0}(F_{\Delta}  - F_{K \times K} - \tfrac{1}{2}[\omega, \omega] + \tfrac{1}{2}[\omega_{K \times K}, \omega_{K \times K}]).$$ Since $[\omega, \omega]$ takes values in $ [\mathfrak{z}_{\Theta}, \mathfrak{z}_{\Theta}] = [\mathfrak{m}_{\Theta} \oplus \mathfrak{a}_{\Theta}, \mathfrak{m}_{\Theta} \oplus \mathfrak{a}_{\Theta}] = [\mathfrak{m}_{\Theta}, \mathfrak{m}_{\Theta}] \subset \mathfrak{m}_{\Theta} $, (cf. Lemma \ref{ma}) and analogously $[\omega_{K \times K}, \omega_{K \times K}]$ takes values in $\mathfrak{m}_{\Theta} \times \mathfrak{m}_{\Theta}$, the result follows by observing that $\langle \mathfrak{a}_{\Theta}, \mathfrak{m}_{\Theta} \rangle = 0$ and $H_{\Theta} \in \mathfrak{a}_{\Theta}$.
\end{proof}
\section{An example: The Lagrangian Grassmanian} \label{sec:lag}
Let $\omega_{0}$ be the standard symplectic struture on $\R^{2n}$, given by $$\omega_{0} = \sum_{i=1}^{n} dx_{i} \wedge dy_{i}$$ where $(x_{1}, \hdots, x_{n}, y_{1}, \hdots, y_{n})$ are the standard cartesian coordinates of $\R^{2n}$. The symplectic group of $(\R^{2n}, \omega_{0})$ consists of all matrices of the form $$\Sp(2n) = \Big\{\begin{pmatrix}
    p & q \\ r & s
\end{pmatrix} \in \GL(2n, \mathbb{R})  : p^{T}s - r^{T}q = 1, \: p^{T}r = r^{T}p, \: q^{T}s =s^{t}q \Big\}. $$
Its maximal compact subgroup is $U(n) = \Sp(2n, \mathbb{R}) \cap O(2n)$ that consists of all matrices of the form $$U(n) = \Big\{ \begin{pmatrix} g & b \\ -b & g \end{pmatrix} \in \GL(2n, \mathbb{R}) : g^{T}g + b^{T}b = I, \: g^{T}b = b^{T}g \Big\}.$$ 
We see that the set of orthogonal matrices is embedded in $U(n)$ via the diagonal map $$O(n) \cong \Delta(O(n)) = \Big\{ \begin{pmatrix}
    g & 0 \\ 0 & g
\end{pmatrix} : g^{T}g=I \Big\} \subset U(n).$$Recall that a subspace $L \subset \R^{2n}$ is called \textit{Lagrangian} if $\omega_{0} \vert_{L \times L} = 0$ and $\dim L = n$. We set  $\LGr_{n} \subset \Gr_{n}(\R^{2n})$ the set of all Lagrangian subspaces of $(\R^{2n}, \omega_{0})$, the so-called \textit{Lagrangian Grassmanian} of $(\R^{2n}, \omega_{0})$. Like the  Grassmanian manifold , the Lagrangian Grassmanian is also a homogeneous space. 
\begin{prop}
The Lagrangian Grassmanian is the homogeneous space $$\LGr_{n} \cong \Sp(2n)/ P \cong U(n)/\Delta(O(n))$$ where $P \subset \Sp(2n)$ is the (maximal) parabolic subgroup of matrices of the form $$P = \Big\{ \begin{pmatrix} g & gb \\ 0 & (g^{-1})^{T}  \end{pmatrix} : g \in \GL(n, \mathbb{R}), b^{T} = b\Big\}.$$  
\end{prop}
\begin{proof}
    First, $\Sp(2n)$ acts transitively on $\LGr_{n}$. To see this, let $L \in \LGr_{n}$ be arbitrary and choose an orthonormal basis $\{u_{i}\}$ of $L$, Let $J$ be the canonical complex structure of $\R^{2n}$ and define $JL = \spam\{v_{i}\}$ where $v_{i}:= Ju_{i}$. This subspace is also Lagrangian  and $L \oplus JL = \R^{2n}$.  Let $A \in \GL(2n, \mathbb{R})$ be the matrix that sends $e_{i} \mapsto u_{i}$ and $f_{i} \mapsto v_{i}$, where $\{e_{i}, f_{i}\}$ is the canonical basis of $\R^{2n}$. A straightforward computation shows that $A$ is holomorphic and orthogonal, whence $A \in U(n) \subset \Sp(2n)$. This shows a stronger claim, that in fact,  $U(n)$ acts transitively on $\LGr_{n}$. To complete the proof we have to compute the isotropy group of the canonical action of $\Sp(2n)$ on $\LGr_{n}$. Let $L_{0} = \R^{n} \oplus \{0\}$ be the horizontal Lagrangian, then $A \in \Sp(2n)$ fixes $L_{0}$ iff $A$ is of the form $$\begin{pmatrix} g & c \\ 0 & h \end{pmatrix}, \qquad g \in \GL(n, \mathbb{R}), \: g^{T}h=I, \:  c^{T}h = h^{T}c.$$ Note that $(h^{T}c)^{T} = c^{T}h = h^{T}c$, hence $b:= h^{T}c$ is symmetric, and $c = (h^{T})^{-1} b = gb.$
\end{proof}
\begin{obs} \label{obs 5}
By our construction, the origin of $\LGr_{n}$ corresponds to the horizontal Lagrangian $L_{0} = \R^{n} \oplus \{0\}$. If we choose to fix the vertical Lagrangian $L_{0}^{\ast} = \{0\} \oplus \R^{n}$ we would obtain the parabolic subgroup $$P^{-} = \Big\{ \begin{pmatrix}(g^{-1})^{T} & 0 \\ gb & g\end{pmatrix} : g \in \GL(n, \mathbb{R}) , b^{T}=b \Big\}$$
Note that $P$ and $P^{-}$ are related via the complex structure $$P^{-} = \widetilde{w}_{0} P \: \widetilde{w}_{0}^{-1}, \qquad \widetilde{w}_{0} = \begin{pmatrix} 0 & I \\ -I & 0 \end{pmatrix} \in U(n)$$ and $\widetilde{w}_{0}$ corresponds exactly to the principal involution $w_{0}$ of the Weyl group of $\mathfrak{sp}(2n, \mathbb{R})$ given by $w_{0} = -\id. $
\end{obs}
The Lie algebra of $\Sp(2n)$ is given by $$\mathfrak{sp}(2n, \mathbb{R}) = \Big\{ \begin{pmatrix} A & B \\ C & -A^{T}  \end{pmatrix} : A \in \mathfrak{gl}(n, \mathbb{R}), \: B = B^{T}, \: C=C^{T} \Big\}.$$ The subset of all diagonal matrices $\mathfrak{a} \subset \mathfrak{sp}(2n, \mathbb{R})$ is a a maximal abelian subalgebra of the symmetric part with respect to the Cartan involution $\theta(X) =-X^{T}$, that is,  $$\mathfrak{s} = \Big\{ \begin{pmatrix}
    A & B \\ B& -A
\end{pmatrix} : A^{T} = A, \: B^{T}=B \Big\} \subset \mathfrak{sp}(2n, \mathbb{R}).$$ Thus, for each $i=1, \hdots, n$, the linear functional  $e_{i} \colon \diag(h_{1}, \hdots, h_{n}, -h_{1}, \hdots, -h_{n}) \in \mathfrak{a} \mapsto h_{i} \in \mathbb{R}$, allow us to write the (restricted) root system of $\mathfrak{sp}(2n, \mathbb{R})$,  $$\Delta = \{\pm e_{i} \pm e_{j} : i<j \} \cup \{\pm2e_{i} : i=1, \hdots, n\}.$$
The set of simple roots is $\Sigma = \{\alpha_{1}, \hdots, \alpha_{n}\}$ where $\alpha_{n} = 2e_{n}$ and $\alpha_{i}  = e_{i} - e_{i+1}$ for $i<n$. We can then represent the Lagrangian Grassmannian as the flag manifold  $$\LGr_{n} = \mathbb{F}_{\Theta}, \qquad \Theta = \Sigma \setminus \{\alpha_{n}\}.$$
A characteristic element $H_{\Theta}$ of $\LGr_{n}$ is given by $$H_{\Theta} = \dfrac{1}{2} \begin{pmatrix}
    I &  0 \\ 0  & -I
\end{pmatrix} $$ so that $$\LGr_{n} = \Ad(U(n)) H_{\Theta} \subset \mathfrak{s}.$$
We identify the subalgebras $\mathfrak{sp}(2n, \mathbb{R})$: \begin{align*}&\mathfrak{k} =  \mathfrak{u(n)=} \Big\{ \begin{pmatrix} A & B \\ - B & A \end{pmatrix} : A^{T}=-A, \: B^{T}=B\Big\},
\\ &\mathfrak{n}^{+}_{\Theta} = \Big\{ \begin{pmatrix}
    0 & B \\ 0 & 0
\end{pmatrix} : B^{T}=B\Big\}, \: \mathfrak{n}_{\Theta}^{-} = \Big\{ \begin{pmatrix}
    0 & 0 \\ C & 0 
\end{pmatrix} : C^{T}=C \Big\},
\\ &\mathfrak{z}_{\Theta} = \Big\{ \begin{pmatrix} A & 0 \\ 0 & -A^{T} \end{pmatrix} : A \in \mathfrak{gl}(n, \mathbb{R}) \Big\}
\\ & \mathfrak{k}_{\Theta} = \Big\{ \begin{pmatrix}
    A & 0 \\ 0 & A
\end{pmatrix} : A^{T}=-A \Big\} = \Delta(\mathfrak{o}(n))
\end{align*}
The $K$-invariant connection $\omega_{K} \in \Omega^{1}(K, \mathfrak{k}_{\Theta}) $  (at the origin) is given by $$\omega_{K} \begin{pmatrix} A & B \\ -B & A  \end{pmatrix} = \proj_{\mathfrak{k}_{\Theta}} \begin{pmatrix} A & B \\ -B & A  \end{pmatrix} = \begin{pmatrix} A & 0 \\ 0 & A  \end{pmatrix}.$$
To compute the extended connection $\omega_{K}^{+} \in \Omega^{1}(G/N_{\Theta}^{+}, \mathfrak{z}_{\Theta})$ we must note that  $$T_{[e, e]} (K \times_{K_{\Theta}} Z_{\Theta} ) \cong (\mathfrak{k} \times \mathfrak{z}_{\Theta})/V, \qquad V = \{(X, -X) : X \in \mathfrak{k}_{\Theta} \}.$$ 
\begin{lema}
    The map $\varphi \colon T_{[e, e]}(K \times_{K_{\Theta}} Z_{\Theta}) \to \mathfrak{sp}(2n, \mathbb{R})/\mathfrak{n}_{\Theta}^{+}$ given by  $$ \varphi   \Big(\Big(\begin{pmatrix} A & B \\ -B & A \end{pmatrix}, \begin{pmatrix}
        C& 0 \\ 0 & -C^{T}
    \end{pmatrix} \Big) + V \Big) =  \begin{pmatrix} A + C & 0 \\ -B & A-C^{T} \end{pmatrix} $$ is an isomorphism of vector spaces.
\end{lema}
\begin{proof} It is straightforward to see that $\varphi$ is a well defined linear map. To see the injectivity pick a generic element \begin{equation} \label{u} u =\Big( \begin{pmatrix} A & B \\ -B & A \end{pmatrix}, \begin{pmatrix}
        C & 0 \\ 0 & -C^{T}
    \end{pmatrix} \Big) \in \mathfrak{k} \times \mathfrak{z}_{\Theta} \end{equation} then $\varphi(u+V) =0$ iff $B=0, A = C^{T}, A= -C$; but this says that $u \in V$, hence $\varphi$ is injective. Since the dimensions of both spaces are equal we have that $\varphi$ is an isomorphism. 
\end{proof}
By identifying $$u = \underbrace{\begin{pmatrix} A + C & 0 \\ -B & A-C^{T}  \end{pmatrix}}_{\in \mathfrak{sp}(2n, \mathbb{R})/\mathfrak{n}_{\Theta}^{+}}\longleftrightarrow \underbrace{\Big( \begin{pmatrix}
    A & B \\ -B & A
\end{pmatrix}, \begin{pmatrix}
    C & 0 \\ 0 & -C^{T}
\end{pmatrix} \Big) + V}_{\in T_{[e, e]}(K \times_{K_{\Theta}} Z_{\Theta})}$$ we can compute the extended connection $$\omega_{K}^{+} (u) = \omega_{K} \begin{pmatrix}
        A & B \\-B & A 
    \end{pmatrix} + \begin{pmatrix}C & 0 \\ 0 & -C^{T}  \end{pmatrix} = \begin{pmatrix} A + C & 0 \\ 0 & A-C^{T} \end{pmatrix} = \proj_{\mathfrak{z}_{\Theta}}(u).$$
     Let us denote by $M$ the adjoint orbit $$M = \Ad(\Sp(2n)) H_{\Theta} \cong T^{\ast} \LGr_{n}.$$ Note that $$T_{H_{\Theta}} M \cong \mathfrak{sp}(2n, \mathbb{R})/\mathfrak{z}_{\Theta} = \mathfrak{n}_{\Theta}^{-} \oplus \mathfrak{n}_{\Theta}^{+}=\Big\{ \begin{pmatrix} 0 & B \\ C & 0 \end{pmatrix}  : B ^{T} = B, \: C^{T}=C \Big\}$$
For an arbitrary element $$U=\begin{pmatrix} A & B \\ C & -A^{T} \end{pmatrix} \in \mathfrak{sp}(2n, \mathbb{R})$$  the fundamental vector field $\widetilde{U} \in \mathfrak{X}(M)$ at the origin $H_{\Theta}$ is given by $$\widetilde{U}(H_{\Theta}) = [U, H_{\Theta}] = \begin{pmatrix}
    0 & -B \\ C & 0
\end{pmatrix}$$ so that the KKS form $\kappa$ in $M$ is  \begin{align*}\kappa_{H_{\Theta}} (\widetilde{U}_{1}, \widetilde{U}_{2}) &= \tr (H_{\Theta} \cdot [U_{1}, U_{2}])  = \tr (B_{1}C_{2}-B_{2}C_{1}). 
\end{align*}
Since the principal involution of $\mathfrak{sp}(2n, \mathbb{R})$ is $w_{0} = -\id$, the dual flag verifies $\Theta^{\ast} = -w_{0}(\Theta) = \Theta$ and therefore we have self duality $\LGr_{n} = \LGr_{n}^{\ast}.$ By identifying the origins of $\LGr_{n}$ and $\LGr_{n}^{\ast}$ with the horizontal and vertical Lagrangians (cf. Remark \ref{obs 5}) we can write the open orbit $$\mathcal{O} = \Sp(2n) \cdot (L_{0}, L_{0}^{\ast}) \subset \LGr_{n} \times \LGr_{n}^{\ast}$$ as $$\mathcal{O} = \{ (L, L') \in \LGr_{n} \times \LGr_{n} : L \cap L' =0 \}$$
which corresponds exactly to the open Bruhat cell once we decompose $$\LGr_{n} \times \LGr_{n} = \bigcup_{k=0}^{n} \Lambda^{k}, \qquad \Lambda^{k} :=\{(L, L') \in \LGr_{n} \times \LGr_{n} : \dim (L \cap L')=k\}.$$
The Lagrangian Grassmannian exhibits, in a particularly transparent form, the interplay between parabolic geometry, connections, and symplectic structures that motivates our general construction. The identification $T^{*}\LGr_{n}\cong\Ad(G)H_{\Theta}$ shows how the cotangent geometry of the flag manifold is encoded in the corresponding adjoint orbit, with the extended connection providing the bridge between the Liouville and KKS descriptions.
\section*{Appendix I: extended connections on principal bundles}\label{app:connection}
In this section we explore a classical construction of extending structural groups and connections on principal bundles. Although some of the results in this section can already be found (essentialy in \cite{kobayashi1996foundations}, \cite{san2021lie}, \cite{hamilton2017mathematical}, \cite{tu2017differential}) we include proofs adapted to the context of our work.
Let $G$ be a Lie group and $H \subset G$ a subgroup. Given a $H$-principal bundle $Q \to M$ we want to build a principal bundle $P \to M$ with structure group $G$, so that $Q$ is a principal subbundle of $P$. As a consequence, every connection form $\omega \in \Omega^{1}(Q, \mathfrak{h})$ has a natural extension $\widetilde{\omega} \in \Omega^{1}(P, \mathfrak{g})$. 
Let $G$ be a Lie group and $H \subset G$ a closed subgroup. Consider a principal $H$-bundle $H \to Q \to M$.
    The manifold $P:=Q \times_{H} G$ can be endowed with a principal bundle structure $$[q, g] \cdot g' :=[q, gg'], \qquad q \in Q, h \in H, g \in G.$$ where the projection $\pi_{P} \colon P \to M$ is given by $\pi_{P}([q, g]) = \pi_{Q}(q).$
    
    Recall that $H$ acts on $Q \times G$ via $$h \cdot (q, g) = (qh, h^{-1}g), \qquad h \in H, (q, g) \in Q \times G.$$
    For every $X \in \mathfrak{h}$, the fundamental vector field of this action $X^{\#} \in \mathfrak{X}(Q \times G)$ is defined by \begin{equation} \label{fund} X^{\#}(q, g) = (\widetilde{X}(q), \dfrac{d}{dt}\Big \vert_{t=0} \exp(-tX) g ) \end{equation} where $\widetilde{X}$ is the fundamental vector field of $H$-right action on $Q$.
    For every $g \in G$, let $R'_{g}, L_{g}' \colon Q \times G \to Q \times G$ denote the maps \begin{align*}R'_{g}(q, h) = (q, hg), \qquad L_{g}'( q, h) &= (q, gh), \qquad (q, h) \in Q \times G. \end{align*}
 Then, if we denote by $$\pi_{H} \colon Q \times H \owns (q, h) \longmapsto [g, h] \in Q \times_{H} G $$ the right action $R_{g} \colon P \to P$  verifies \begin{equation} \label{Rg} R_{g}([q, h]) = [q, hg] = \pi_{H}(R_{g}'(q, h)) \iff R_{g} \circ \pi_{H} = \pi_{H} \circ R_{g}' \end{equation}
    Suppose that $\omega \in \Omega^{1}(Q, \mathfrak{h})$ is a connection defined on $Q$. 
     We shall consider the $\mathfrak{g}$-valued form $\widetilde{\omega} \in \Omega^{1}(Q \times G, \mathfrak{g})$ given by \begin{equation} \label{omegatilde}\widetilde{\omega}_{(q, g)}(u, v) :=\Ad(g)^{-1}\omega_{q}(u) + \theta_{g} (v), \qquad (u, v) \in T_{(q, g)}(Q \times G). \end{equation}
     where $\theta \in \Omega^{1}(G, \mathfrak{g})$ is the Maurer-Cartan form, $$\theta_{g}(v) = d(l_{g^{-1}})_{g}v, \qquad v \in T_{g}G.$$ 
 where $l, r \colon G \to G$ denote the left and right translations, respectively. We recall some well known properties of the Maurer-Cartan form. \begin{enumerate}[(i)]
        \item (\textit{Left-invariance}) $l_{g}^{\ast} \theta = \theta$, for all $g \in G.$
        \item (\textit{Pseudotensoriality}) $r_{g}^{\ast} \theta = \Ad(g^{-1}) \theta$, for all $g \in G$.
        \item (\text{Maurer-Cartan equation}) $d\theta + \tfrac{1}{2}[\theta, \theta] = 0.$
    \end{enumerate}
    \begin{lema} \label{ad1}
        For all $g \in G$ we have \begin{equation} 
            (R_{g}')^{\ast} \widetilde{\omega} = \Ad(g^{-1}) \widetilde{\omega}
        \end{equation}
    \end{lema}
    
    \begin{lema}  \begin{enumerate}[(i)] $\widetilde{\omega}$ is horizontal and $H$-invariant, that is 
         \item $\widetilde{\omega}(X^{\#}) = 0$ for every $H$-fundamental vector field $X^{\#}$ as defined in (\ref{fund}). 
        \item $\widetilde{\omega}$ is $H$-invariant. 
        \end{enumerate}
        Thus, there is a unique form $\omega^{e} \in \Omega^{1}(P, \mathfrak{g})$ such that \begin{equation} \label{comu} \pi_{H}^{\ast} \omega^{e} = \widetilde{\omega} \end{equation}
    \end{lema}
    \begin{proof}
        To prove $(i)$ pick any $X \in \mathfrak{h}$ and write \begin{align*}\widetilde{\omega}(X^{\#}) &= \Ad(g)^{-1} \omega(\widetilde{X}) + \theta \Big(\dfrac{d}{dt} \Big\vert_{t=0} \exp (-tX)g \Big) \\ & = \Ad(g)^{-1} X + \theta (d(R_{g})_{e}(-X)) \\& = \Ad(g)^{-1}X  - \Ad(g^{-1})X =0.\end{align*} For $(ii)$ let $q \in Q, h \in H$ and $g \in G$, we must show that $$\widetilde{\omega}_{(qh, h^{-1}g)}(u, v) = \widetilde{\omega}_{(q, g)}(u', v'), \quad \begin{cases}
        u = (r_{h})_{\ast}u' \\ v = (l_{h^{-1}})_{\ast} v'
    \end{cases}$$ Note that  \begin{align*}
        \omega_{(qh, h^{-1}g)}(u, v) &= \Ad(h^{-1}g)^{-1} \omega_{qh}(u) + \theta_{h^{-1}g}(v).
    \end{align*}
Now, $\theta_{h^{-1}g}(v) = \theta_{g}(v')$ because $\theta$ is left-invariant, and $\omega_{qh}(u) = \Ad(h^{-1})\omega_{q}(u')$ by Lemma \ref{ad1}, therefore the assertion follows. 
    \end{proof}
    \begin{prop} \label{lema 6}
        The form $\omega^{e} \in \Omega^{1}(P, \mathfrak{g})$ defined in (\ref{comu}) is a connection on the $G$-principal bundle $P$.
    \end{prop}
    \begin{proof}
         By  Equations (\ref{comu}), (\ref{Rg}) and (\ref{ad1}) we have \begin{align*} \pi_{H}^{\ast}
         R_{g}^{\ast}\omega^{e} &= (R_{g} \circ \pi_{H})^{\ast} \omega^{e} = (\pi_{H} \circ R'_{g})^{\ast} \omega^{e} = (R_{g}')^{\ast} \pi_{H}^{\ast} \omega^{e} \\ &= (R_{g}')^{\ast} \widetilde{\omega} = \Ad(g^{-1})\widetilde{\omega} = \Ad(g^{-1}) \pi_{H}^{\ast} \omega^{e} = \pi_{H}^{\ast} (\Ad(g^{-1}) \omega^{e}) \end{align*} that is $$R_{g}^{\ast} \omega^{e} = \Ad(g^{-1}) \omega^{e}.$$ 
         Now, let $X\in\mathfrak g$. 
The fundamental vector field of the right $G$-action on $P$ is given by 
\begin{align*} X^\#_{[q,g]} &=
\frac{d}{dt}\Big|_{t=0} R_{\exp tX} (\pi_{H}(q, g)) = \dfrac{d}{dt}\Big\vert_{t=0}\pi_{H}(q,g\exp(tX)) = d\pi_{H} \: (0, d(l_{g})_{e}X)  \end{align*}
Whence $$\omega^{e}(X^{\#}([q, g])) = \theta_{g}(d(l_{g})_{e}X) = X$$ and thus  $\omega^{e}$ reproduces the generators of the fundamental vector fields, and the proof is concluded.
    \end{proof}
    \begin{lema} \label{dif} For any $X \in \mathfrak{g}$, define  $f \colon G \to \mathfrak{g}$ by $f(g) = \Ad(g^{-1})X$. Then its differential is $$df_{g}(v) = [\Ad(g^{-1})X, \theta(v)], \qquad v \in T_{g}G.$$
\end{lema}
\begin{proof}
    Let $g \in G$ and $v \in T_{g}G$, so that $v = (L_{g})_{\ast} Y$ where $Y \in \mathfrak{g}$, note that $\theta(v) = Y.$ Define $g(t) = g\exp tY \in G$ and compute \begin{align*}df_{g}(v) &= \dfrac{d}{dt} \Big\vert_{t=0} f(g(t)) \\&= \dfrac{d}{dt} \Big\vert_{t=0}  \Ad( \exp(-tY)) \Ad(g^{-1}) X \\&= [\Ad(g^{-1})X, Y] \\&= [\Ad(g^{-1})X, \theta(v)].
    \end{align*}
\end{proof}
    \begin{lema}
    If $\omega \in \Omega^{1}(P, \mathfrak{g})$ define $A \in \Omega^{1}(P \times G, \mathfrak{\mathfrak{g}})$ by $A_{(p, g)} := \Ad(g^{-1}) \omega_{p}$. Then $$dA_{(p, g)} = \Ad(g^{-1}) d\omega_{p} - [A, \theta]_{p}.$$
\end{lema}
\begin{proof}
   Write locally $\omega = \sum_{i} \omega_{i} E_{i}$, where $E_{i}$ is a basis of $\mathfrak{g}$ and $\omega_{i}$ is a one-form locally defined on $P$. Then $A = \sum \omega_{i} A_{i}$ where $A_{i} \colon G \to \mathfrak{g} $ is given by $A_{i}(g)=\Ad(g^{-1}) E_{i}$, therefore $d(\omega_{i} A_{i}) = d\omega_{i} A_{i}
    - \omega_{i} \wedge dA_{i}$ and $$dA = \sum_{i} d\omega_{i} A_{i} -\sum_{i} \omega_{i} \wedge dA_{i}.$$ Now, by Lemma \ref{dif} we have $$dA_{i} = [\Ad(g^{-1})E_{i}, \theta] \in \Omega^{1}(G, \mathfrak{g})$$ and $(\omega_{i} \wedge dA_{i})(X, Y) = \omega_{i}(X) dA_{i}(Y) - \omega_{i}(Y) dA_{i}(X)$ whence \begin{align*}\Big(\sum_{i} \omega_{i} \wedge dA_{i} \Big)(X, Y) &= \sum_{i} \omega_{i}(X) \ [A_{i}, \theta(Y)] - \omega_{i}(Y)[A_{i}, \theta(X)] \\&=\Big[\sum_{i} \omega_{i} A_{i}, \theta \Big](X, Y) \\&= [A, \theta](X, Y).
    \end{align*}
\end{proof}
\begin{teo} \label{most}
    If $F \in \Omega^{2}(Q, \mathfrak{h})$ denotes the curvature of the connection $\omega$ and $F^{e} \in \Omega^{2}(P, \mathfrak{g})$ denotes the curvature of $\omega^{e}$, we have $$F^{e}_{[q, g]} = \Ad(g^{-1}) F_{q}.$$
\end{teo}
\begin{proof}
    Let us define $\widetilde{F} := d\widetilde{\omega} + \tfrac{1}{2}[\widetilde{\omega}, \widetilde{\omega}] \in \Omega^{2}(Q \times G, \mathfrak{g})$, where $\widetilde{\omega}$ was defined in \ref{omegatilde}. By the structure equation, $$\pi_{H}^{\ast} F^{e} = \widetilde{F}.$$  But, $$[\widetilde{\omega}, \widetilde{\omega}] = [A + \theta, A+\theta] = [A, A] + 2[A, \theta] + [\theta, \theta]$$ where $A \in \Omega^{1}(Q \times G, \mathfrak{g})$ is given by $A_{(q, g)} = \Ad(g^{-1}) \omega_{q}$. Now, by the previous Lemma we have $$d\widetilde{\omega}_{(q, g)} = \Ad(g^{-1})d\omega_{q} - [A, \theta]_{q} + d\theta_{g}$$ whence by the Maurer-Cartan equation we conclude  $$\widetilde{F}_{(q, g)} = \Ad(g^{-1}) (d\omega_{q} + \tfrac{1}{2}[\omega, \omega]_{q}) = \Ad(g^{-1})F_{q}. $$
\end{proof}
\begin{obs} \label{ext} 
The principal bundle $H \to Q \to M$ embeds naturally on the $G$-principal bundle $Q \times_{H} G$ via  $$Q \hookrightarrow Q \times G \stackrel{\pi_{H}}{\longrightarrow} Q \times_{H} G.$$ Hence we shall regard the connection $\widetilde{\omega} \in \Omega^{1}(Q \times_{H} G, \mathfrak{g})$ of Lemma \ref{lema 6} as an extension of the connection $\omega \in \Omega^{1}(Q, \mathfrak{h})$. 
\end{obs}

\section*{Appendix II: The dual flag $\mathbb{F}_{\Theta}^{\ast}$} \label{appendix 2}
The Weyl group $\mathcal{W}$ of the restricted root system $\Delta(\mathfrak{g}, \mathfrak{a})$ has a distinguished element $w_{0} \in \mathcal{W}$ characterized by the property that $w_{0}^{2} = 1$ and $-w_{0}(\Sigma) = \Sigma.$ Such $w_{0}$ is called a \textit{principal involution} of $\mathcal{W}$. Given $\Theta \subset \Sigma$ we define its dual $\Theta^{\ast} \subset \Sigma$ by $$\Theta^{\ast} := -w_{0} \Theta.$$ 
Since the Weyl group can be realized as quotient $$\mathcal{W} = N_{K}(\mathfrak{a})/Z_{K}(\mathfrak{a})$$ we shall put $\widetilde{w} \in N_{K}(\mathfrak{a})$ a representative of $w \in \mathcal{W}$.
 \\ Note that $\alpha \in \Theta^{\ast}$ if and only if there exists a unique $\beta \in \Theta$ with $-\alpha = w_{0} \beta$, hence since $\Ad(\widetilde{w}) \g_{\alpha} = \g_{w \alpha}$ we have  \begin{equation*}
        \sum_{\alpha \in \langle \Theta^{\ast} \rangle^{+}} \mathfrak{g}_{-\alpha} = \sum_{\beta \in \langle \Theta \rangle^{+}} \mathfrak{g}_{w_{0}\beta} = \Ad(\widetilde{w}_{0}) \Big( \sum_{\beta \in \langle \Theta \rangle^{+}} \mathfrak{g}_{\beta} \Big).
    \end{equation*}
    If $\mathfrak{b}^{-} = \mathfrak{m} \oplus \mathfrak{a} \oplus \mathfrak{n}^{-}$ and $\mathfrak{p}_{\Theta}^{-} = \mathfrak{n}_{\Theta}^{-} \oplus \mathfrak{z}_{\Theta}$ we have \begin{equation*} \label{important one} \mathfrak{p}_{\Theta^{\ast}}^{+} = \Ad(\widetilde{w}_{0}) (\mathfrak{b}^{-} \oplus \sum _{\alpha \in \langle \Theta \rangle^{+}} \mathfrak{g}_{\alpha}) = \Ad(\widetilde{w}_{0}) \mathfrak{p}_{\Theta}^{-},\end{equation*}
so it follows that
\begin{equation} \label{important} P_{\Theta^{\ast}}^{+} = \widetilde{w}_{0} P_{\Theta}^{-} \widetilde{w}_{0}^{-1} = \widetilde{w}_{0} P_{\Theta}^{-} \widetilde{w}_{0}.\end{equation}
The \textit{dual flag} of $\mathbb{F}_{\Theta}$ is defined as the flag corresponding to $\Theta^{\ast}$,
$$\mathbb{F}_{\Theta}^{\ast}:= \mathbb{F}_{\Theta^{\ast}} = G/P_{\Theta^{\ast}}^{+}.$$
Denote by $o:=eP_{\Theta}^{+}$ the origin of the flag $\mathbb{F}_{\Theta}$, and set
$$o_{w_{0}} := \widetilde{w}_{0} P_{\Theta^{\ast}}^{+} \in \mathbb{F}_{\Theta}^{\ast},$$
so that $\mathbb{F}_{\Theta}^{\ast} = G \cdot o_{w_{0}}$. The point $o_{w_{0}}$ is usually be referred as the origin of the dual flag $\mathbb{F}_{\Theta}^{\ast}$.
\begin{lema}
    $G$ acts diagonally on the product $\mathbb{F}_{\Theta} \times \mathbb{F}_{\Theta}^{\ast}$ and we have a diffeomorphism $$G \cdot (o, o_{w_{0}}) := \{(go, go_{w_{0}}) : g \in G\} \cong G/Z_{\Theta}.$$ Moreover, the orbit $G \cdot (o, o_{w_{0}}) \subset \mathbb{F}_{\Theta} \times \mathbb{F}_{\Theta}^{\ast}$ is open and dense. 
\end{lema}
\begin{proof}In order to prove the identification we just compute the stabilizer of $(o, o_{w_{0}})$ which is $$\Stab_{G}(o, o_{w_{0}}) = \Stab_{G}(o) \cap \Stab_{G}(o_{w}).$$ 
Now, $\Stab_{G}(o) = P_{\Theta}^{+}$ and $g \in \Stab_{G}(o_{w})$ if and only if $g \widetilde{w}_{0} P_{\Theta^{\ast}}^{+} = \widetilde{w}_{0}P_{\Theta^{\ast}}^{+}$ thus $\widetilde{w}_{0}^{-1} g \widetilde{w}_{0} \in P_{\Theta^{\ast}}$ which, by (\ref{important}), only happens if and only if $g \in P_{\Theta}^{-}$. Hence $\Stab_{G}(o_{w}) = P_{\Theta}^{-}$ and $\Stab_{G}(o, o_{w}) = P_{\Theta}^{+} \cap P_{\Theta}^{-} = Z_{\Theta}$ and the identification follows. 
\end{proof}

\section{Acknowledgments}
J.F. is supported by The São Paulo Research Foundation (FAPESP) grant 2025/22423-8. L. G. is partially supported by The São Paulo Research Foundation (FAPESP) grants 2023/13131-8 and 2024/00923-6; and by National Council for Scientific and Technological Development (CNPq) grant 306021/2024-2.

\bibliographystyle{plain}
\bibliography{mybib}
\end{document}